%% file: body.tex
\documentclass{article}

\usepackage{maa-monthly} 
\trimmarksfalse

\input{package}

\input{header}

\input{notation}

\begin{document}

\title{On the Optimality of Napoleon Triangles}
\markright{On the Optimality of Napoleon Triangles }
\author{Omur Arslan\thanks{The authors are with the Department of Electrical and Systems Engineering, University of Pennsylvania, Philadelphia, PA 19104. E-mail: \{omur, kod\}@seas.upenn.edu} ~and Daniel E. Koditschek$^*$ }

\maketitle


\begin{abstract}
An elementary geometric construction known as Napoleon's theorem produces an equilateral triangle built on the sides of any  initial triangle: 
the centroids of each equilateral triangle meeting the original sides, all outward or all inward, comprise the vertices of the new equilateral triangle. 
In this note we observe that two Napoleon iterations yield triangles with useful optimality properties. 
Two inner transformations result in a (degenerate) triangle whose  vertices coincide at the original centroid. 
Two outer transformations yield an  equilateral triangle whose vertices are closest to the original in the sense of minimizing the sum of the three squared distances. 
\end{abstract}

\section{Introduction.}
\label{sec:Introduction}

In elementary geometry, one way of constructing an equilateral triangle from any given triangle is as follows: in a plane the centroids of  equilateral triangles erected, either all externally or all internally, on the sides of the given triangle form an equilateral triangle, illustrated in \reffig{fig.FermatTorricelliNapoleon} \cite{coxeter1996geometry}.
This result is generally referred to as \emph{Napoleon's theorem}, notwithstanding its dubious origins --- see \cite{grunbaum_AMM2012} for a detailed history of the theorem.
We will refer to these constructions as \emph{the outer and inner Napoleon transformations} and the associated equilateral triangles as \emph{the outer and inner Napoleon triangles} of the original triangle, respectively.
Conversely,  given  its outer and inner Napoleon triangles in position (i.e. they are oppositely oriented and   have the same centroid), the original triangle is uniquely determined \cite{wetzel_AMM1992}.
A fascinating application of Napoleon triangles is the planar tessellation used by Escher: a plane can be tiled using congruent copies of the hexagon defined by the vertices of  any triangle and its outer Napoleon triangle, known as \emph{Escher's theorem} \cite{rigby_MM1991}.    

\begin{figure}[h]
\centering
\includegraphics[width=0.8\textwidth]{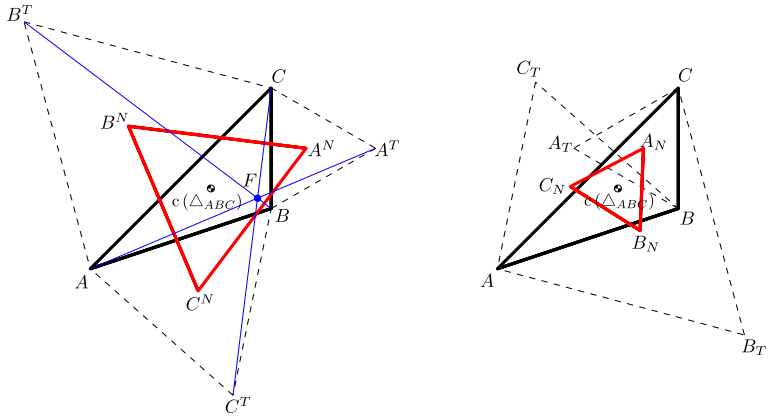}
\caption{An illustration of (left) the Fermat point $F$, outer Torricelli configuration $\triangle_{A^TB^TC^T}$ and outer Napoleon triangle  $\triangle_{A^NB^NC^N}$, and (right) inner Torricelli configuration $\triangle_{A_TB_TC_T}$ and inner Napoleon triangle $\triangle_{A_NB_NC_N}$ of a triangle $\triangle_{ABC}$. 
Note that centroids of Torricelli configurations, Napoleon triangles and the original triangle all coincide, i.e. $\ctrd{\triangle_{ABC}} = \ctrd{\triangle_{A^TB^TC^T}} = \ctrd{\triangle_{A^NB^NC^N}} = \ctrd{\triangle_{A_TB_TC_T}} = \ctrd{\triangle_{A_NB_NC_N}}$.}
\label{fig.FermatTorricelliNapoleon}
\end{figure}

Equilaterals built on the sides of a triangle make a variety of appearances in the classical literature. 
Torricelli uses this construction to locate Fermat's point minimizing the sum of distances to the vertices of a given triangle, now called \emph{the Fermat-Torricelli problem} \cite{kupitz_martini_spirova_JOTA2013}.
The unique solution of this problem is known as \emph{the Fermat-Torricelli point} of the given triangle, located as follows \cite{gueron_tessler_AMM2002}.
If an internal angle of the triangle is greater than $120^\circ$, then the Fermat point is at that obtuse vertex.
Otherwise,  the three lines joining  opposite vertices of the original triangle and externally erected triangles are concurrent, and they intersect at the Fermat point, see \reffig{fig.FermatTorricelliNapoleon}.
The triangle defined by the new vertices of the erected equilateral triangles is referred  to as \emph{the Torricelli configuration} \cite{martini_weissbach_GD1999, hajja_martini_spirova_BAG2008}. 

In this paper we demonstrate some remarkable, but not immediately obvious, optimality properties of twice iterated Napoleon triangles.
First, two composed inner Napoleon transformations of a triangle collapse the original to a point located at its centroid which, by definition, minimizes the sum of squared distances to the vertices of the given triangle.           
Surprisingly, two composed outer Napoleon transformations yield an equilateral triangle optimally aligned with the original by virtue of minimizing the sum of squared distances between the paired vertices.
More precisely, for any triangle $\triangle_{ABC}$ with the vertices \footnote{Here, $\R$ denotes the set of real numbers, and $\R^d$ is the $d$-dimensional Euclidean space.} $A,B,C \in \R^d$,  we will say that the triangle $\triangle_{A'B'C'}$ is  an { \em optimally aligned } equilateral triangle of $\triangle_{ABC}$ if it solves the following constrained optimization problem:  
\begin{equation}\label{eq.Optimization}
\begin{split}
\text{minimize } & \quad \norm{A - A'}^2 + \norm{B - B'}^2 + \norm{C - C'}^2 \\
\text{subject to}& \quad \norm{A' - B'}^2 = \norm{A'-C'}^2 = \norm{B' - C'}^2  
\end{split}
\end{equation}
where $A',B', C' \in \R^d$ and $\norm{.}$ denotes the standard Euclidean norm on $\R^d$.
As we show below, this optimization problem has a unique solution so long as $A$, $B$, $C$ are not collinear.

\section{Torricelli and Napoleon Transformations.}
\label{sec:Notation}

For any ordered triple $\vectbf{x} = \tr{\threevecT{\vect{x}_1}{\vect{x}_2}{\vect{x}_3}} \in \R^{3d}$, let $\RotMat_{\vectbf{x}}$ denote the rotation matrix corresponding to a counter-clockwise rotation by $\pi/2$ in the plane, defined by orthonormal vectors $\vect{n}$ and  $\vect{t}$, in which the triangle $\triangle_{\vectbf{x}}$ formed by $\vectbf{x}$ is positively oriented (i.e. its vertices in counter-clockwise order follow the sequence $\ldots \rightarrow 1\rightarrow 2 \rightarrow 3 \rightarrow 1 \rightarrow \ldots$), \footnote{$\tr{\mat{A}}$ denotes the transpose of matrix $\mat{A}$.} 
\begin{equation}\label{eq.RotMat}
\RotMat_{\vectbf{x}} \ldf \twovect{\vect{n}}{\vect{t}} 
\brl{
\begin{array}{rr}
0 & -1 \\ 1 & 0
\end{array}}
\tr{\twovect{\vect{n}}{\vect{t}}}, 
\end{equation}
where \footnote{For any trivial triangle $\triangle_\vectbf{x}$ all of whose vertices are located at the same point we fix $\RotMat_{\vectbf{x}} = \mat{0}$ by setting $\frac{\vect{x}}{\norm{\vect{x}}} = \vect{0}$ whenever $\vect{x} = 0$.}
\begin{equation} \label{eq.nt}
\vect{n} \ldf \! \left \{ 
\begin{array}{@{}l@{}@{}l@{}}
\frac{\vect{x}_2 - \vect{x}_1}{\norm{\vect{x}_2 - \vect{x}_1}_2} & \text{, if } \vect{x_1} \neq \vect{x}_2,\\ 
\frac{\vect{x}_3 - \vect{x}_2}{\norm{\vect{x}_3 - \vect{x}_2}_2} & \text{, otherwise,} 
\end{array}
\right. \,\,\,
\vect{t} \ldf \! \left \{ 
\begin{array}{@{}l@{}@{}l@{}}
\in \crl{\vect{z} \in \Sp^{d-1} | \vectprod{\vect{n}}{\vect{z}} = 0 } & \text{, if  $\vectbf{x}$ is collinear},\\
 \TpS{\vect{n}}   \frac{\vect{x}_3 - \vect{x}_1}{\norm{\vect{x}_3 - \vect{x}_1}_2} & \text{, otherwise,}
\end{array}
\right. \!\!
\end{equation}
where  $\TpS{\vect{n}} \ldf \mat{I}_{d} - \vect{n}\tr{\vect{n}}$ is the projection onto  $T_{\vect{n}}\Sp^{d-1}$ (the tangent space of $\Sp^{d-1}$ at point $\vect{n} \in \Sp^{d-1}$), and $\mat{I}_{d}$ is the $d \times d$ identity matrix.
Note that $\triangle_{\vectbf{x}}$ is both positively and negatively oriented if $\vectbf{x}$ is collinear. 
Consequently, to define a plane containing such $\vectbf{x}$ we select an arbitrary vector $\,\vect{t}$  perpendicular to $\,\vect{n}$ in \refeqn{eq.nt}.
It is also convenient to have $\ctrd{\vectbf{x}}$ denote the centroid of $\triangle_\vectbf{x}$, i.e. $\ctrd{\vectbf{x}} \ldf \frac{1}{3}\sum_{i =1}^{3} \vect{x}_i$.

In general, the Torricelli and Napoleon transformations of three points in  Euclidean $d$-space can be defined based on their original planar definitions in a 2-dimensional subspace of $\R^d$ containing $\vectbf{x}$.
That is to say, for any  $\vectbf{x} \in \R^{3d}$, select a 2-dimensional subspace of $\R^d$ containing $\vectbf{x}$, and then construct the erected triangles on the side of $\triangle_{\vectbf{x}}$ in this subspace to obtain the Torricelli and Napoleon transformations of $\vectbf{x}$.
Accordingly, let $\TT{\pm}{}:\R^{3d} \rightarrow \R^{3d}$ and $\NT{\pm}{}:\R^{3d} \rightarrow \R^{3d}$ denote the Torricelli and Napoleon transformations where  the sign, $+$ and $-$, determines the type of the transformation, inner  and outer, respectively.
One can write  closed-form expressions of the Torricelli and Napoleon transformations as:

\begin{lemma}
The Torricelli and Napoleon transformations of any triple $\vectbf{x} \in \R^{3d}$ on a plane containing $\vectbf{x}$ are, respectively, given by \footnote{Here, $\otimes$ denotes the Kronecker product \cite{Graham_1981}.} 

\noindent
\begin{align}
\TT{\pm}{}\prl{\vectbf{x}} &= \prl{\frac{1}{2}\mat{K} \pm \frac{\sqrt{3}}{2} \prl{\mat{I}_3 \otimes \RotMat_{\vectbf{x}}}\mat{L}} \vectbf{x}, \label{eq.Torricelli}\\
\NT{\pm}{}\prl{\vectbf{x}} &= \frac{1}{3}\prl{\mat{K}\vectbf{x} + \TT{\pm}{}\prl{\vectbf{x}}}, 
\label{eq.Napoleon}
\end{align}
where
\begin{equation} \label{eq.KLmat}
\mat{K} = \brl{\begin{array}{@{\hspace{1mm}}c@{\hspace{2mm}}c@{\hspace{2mm}}c@{\hspace{1mm}}}
0 & 1 & 1 \\
1 & 0 & 1 \\
1 & 1 & 0
\end{array}} \otimes \mat{I}_d 
\quad  
\text{ and }
\quad 
\mat{L} = \brl{\begin{array}{@{}rrr@{}}
0 & -1 & 1 \\
1 & 0 & -1 \\
-1 & 1 & 0
\end{array}} \otimes \mat{I}_d. 
\end{equation} 
\end{lemma}
\begin{proof}
One can locate the new vertex of an  equilateral triangle, inwardly or outwardly, constructed on one side of $\bigtriangleup_{\vectbf{x}}$ in the plane containing $\vectbf{x}$ using different geometric properties of equilateral triangles.
We find it convenient to use the perpendicular bisector of the  corresponding side of $\bigtriangleup_{\vectbf{x}}$, the line   passing through its midpoint and being perpendicular to it, such that  the new vertex is on this bisector and a proper distance away from the side  of $\triangle_{\vectbf{x}}$.

For instance, let $\vectbf{y} = \tr{\brl{\vect{y}_1, \vect{y}_2, \vect{y}_3}} = \TT{+}{}\prl{\vectbf{x}}$. 
Consider the side of $\triangle_\vectbf{x}$ joining $\vect{x}_1 $ and $\vect{x}_2$, using the bisector, $\vect{b}_{12} \ldf \frac{1}{2} \prl{\vect{x}_1 + \vect{x}_2} $, to locate the new vertex, $\vect{y}_3$, of inwardly erected triangle on this side  as
\begin{equation}
\vect{y}_3 = \vect{b}_{12} + \frac{\sqrt{3}}{2}\RotMat_{\vectbf{x}}\prl{\vect{x}_2 - \vect{x}_1},
\end{equation}
where $\RotMat_{\vectbf{x}}$ \refeqn{eq.RotMat} defines a counter-clockwise rotation by $\frac{\pi}{2}$ in the plane where $\vectbf{x}$ is positively oriented.
Note that the height of an equilateral triangle from any side is $\frac{\sqrt{3}}{2}$ times its side length. 
Hence, by symmetry, one can conclude \refeqn{eq.Torricelli}. 

Given a Torricelli configuration $\vectbf{y} = \tr{\brl{\vect{y}_1, \vect{y}_2, \vect{y}_3}} = \TT{\pm}{}\prl{\vectbf{x}}$, by definition, the vertices of associated Napoleon triangle $\vectbf{z} = \tr{\brl{\vect{z}_1, \vect{z}_2, \vect{z}_3}} = \NT{\pm}{}\prl{\vectbf{x}}$ are given by
\begin{equation}
\vect{z}_1 = \frac{1}{3}\prl{\vect{y}_1 + \vect{x}_2 + \vect{x}_3}, \,
\vect{z}_2 = \frac{1}{3}\prl{\vect{x}_1 + \vect{y}_2 + \vect{x}_3}  \text{ and } \,
\vect{z}_3 = \frac{1}{3}\prl{\vect{x}_1 + \vect{x}_2 + \vect{y}_3},
\end{equation}
which is equal to \refeqn{eq.Napoleon}, and so the result follows. 
\end{proof}
Note that the Torricelli and Napoleon transformations of $\vectbf{x}$ are unique  if and only if $\vectbf{x} \in \R^{3d}$ is non-collinear. 
If, contrarily, $\vectbf{x}$ is collinear, then $\triangle_\vectbf{x}$ is both positively and negatively oriented and  for $d \geq 3$ there is more than one 2-dimensional subspace of $\R^d$ containing $\vectbf{x}$.

\begin{remark}[\hspace{-0.1mm}\cite{wetzel_AMM1992}]
For any $\vectbf{x}=\tr{\brl{\vect{x}_1, \vect{x}_2, \vect{x}_3}} \in \R^{3d}$, the centroid of the Torricelli configuration 
 $\vectbf{y} = \tr{\brl{\vect{y}_1, \vect{y}_2, \vect{y}_3}} = \TT{\pm}{}\prl{\vectbf{x}}$,  the Napoleon configuration $\vectbf{z} = \NT{\pm}{}\prl{\vectbf{x}}$ and the original triple $\vectbf{x}$ all coincide,
\begin{equation}
\ctrd{\vectbf{x}} = \ctrd{\vectbf{y}} = \ctrd{\vectbf{z}},
\end{equation}
and  distances between the associated elements of $\vectbf{x}$ and $\vectbf{y}$ are all the same, i.e.   for any $i \neq j \in \crl{1,2,3}$ 
\begin{equation}
\norm{\vect{y}_i - \vect{x}_i}_2 = \norm{\vect{y}_j - \vect{x}_j}_2.
\end{equation}
\end{remark}

An observation key to all further results is that Napoleon transformations of equilateral triangles are very simple. 
\begin{lemma}\label{lem.NapoleonEquilateral}
The inner Napoleon transformation $\NT{+}{}$ of any  triple $\vectbf{x}= \tr{\brl{\vect{x}_1, \vect{x}_2, \vect{x}_3}} \in \R^{3d}$ comprising the vertices of an equilateral triangle $\triangle_{\vectbf{x}}$ collapses it to the trivial triangle all of whose vertices are located at its centroid $\ctrd{\vectbf{x}}$,
\begin{equation} \label{eq.InnerNapoleonEquilateral}
\NT{+}{}\prl{\vectbf{x}} = \mat{1}_3 \otimes \ctrd{\vectbf{x}},
\end{equation}
whereas the outer Napoleon transformation $\NT{-}{}$ reflects the vertices of $\triangle_{\vectbf{x}}$ with respect to its centroid $\ctrd{\vect{x}}$, \footnote{Here, $\mat{1}_3$ is the $\R^3$ column vector of all ones, and $\cdot$ denotes the standard array product.}
\begin{equation}\label{eq.NapoleonReflection}
\NT{-}{}\prl{\vectbf{x}} = 2 \cdot \mat{1}_3 \otimes \ctrd{\vectbf{x}} - \vectbf{x}. 
\end{equation}
\end{lemma}
\begin{proof}
Observe that the inwardly erected triangle on any side of an equilateral triangle is equal to the equilateral triangle itself, i.e.  $ \TT{+}{}\prl{\vectbf{x}} = \vectbf{x}$, and so, by definition, one has \refeqn{eq.InnerNapoleonEquilateral}.
Alternatively, using \refeqn{eq.Napoleon}, one can obtain
\begin{equation}
\NT{+}{}\prl{\vectbf{x}} = \frac{1}{3}\prl{\mat{K}\vectbf{x} + \TT{+}{}\prl{\vectbf{x}}} = \frac{1}{3}\prl{\mat{K}\vectbf{x} + \vectbf{x}} = \mat{1}_3 \otimes \ctrd{\vectbf{x}},
\end{equation}
where $\mat{K}$ is defined as in \refeqn{eq.KLmat}.

Now consider   outwardly erected equilateral triangles on the sides of an equilateral triangle, and let $\vectbf{y} = \tr{\brl{\vect{y}_1, \vect{y}_2, \vect{y}_3}} = \TT{-}{}\prl{\vectbf{x}}$.
Note that each erected triangle has a common side with the original triangle. 
Since $\triangle_{\vectbf{x}}$ is equilateral, observe that the midpoint of the unshared vertices of an erected triangle and the original triangle is equal to the midpoint of their common sides, i.e. $\frac{1}{2}\prl{\vect{y}_1 + \vect{x}_1} = \frac{1}{2}\prl{\vect{\vect{x}_2} + \vectbf{x}_3}$ and so on.
Hence, we have $\TT{-}{}\prl{\vectbf{x}} = \mat{K}\vectbf{x} - \vectbf{x}$.
Thus, one can verify the result using \refeqn{eq.Napoleon} as
\begin{equation}
\NT{-}{}\prl{\vectbf{x}} = \frac{1}{3}\prl{\mat{K}\vectbf{x} + \TT{-}{}\prl{\vectbf{x}}}= \frac{1}{3}\prl{\mat{K}\vectbf{x} + \mat{K}\vectbf{x} - \vectbf{x}} = 2 \cdot \mat{1}_3 \otimes \ctrd{\vectbf{x}} - \vectbf{x}. \qedhere
\end{equation}
\end{proof}

Since the Napoleon transformation of any triangle results in an equilateral triangle, motivated from \reflem{lem.NapoleonEquilateral}, we now consider  the iterations of the Napoleon transformation. 
For any $k\geq 0$ let $\NT{\pm}{k}: \R^{3d} \rightarrow \R^{3d}$ denote the $k$-th Napoleon transformation defined to be
\begin{equation}
\NT{\pm}{k+1} \ldf \NT{\pm}{} \circ \NT{\pm}{k},
\end{equation}
where we set $\NT{\pm}{0} \ldf id$, and $id:\R^{3d} \rightarrow \R^{3d}$ is the identity map on $\R^{3d}$.

\smallskip

It is evident from  \reflem{lem.NapoleonEquilateral} that:
\begin{lemma} \label{lem.NapoleonIteration}
For any $\vectbf{x} \in \R^{3d}$ and $k\geq 1$,
\begin{equation}
\NT{+}{k+1}\prl{\vectbf{x}} = \mat{1}_3 \otimes \ctrd{\vectbf{x}}, \quad  \text{ and } \quad 
\NT{-}{k+2}\prl{\vectbf{x}} = \NT{-}{k}\prl{\vectbf{x}}.
\end{equation}
\end{lemma}

\noindent As a result, the basis of iterations of the Napoleon transformations consists of
$\NT{\pm}{}$ and $\NT{\pm}{2}$, whose explicit forms, except $\NT{-}{2}$, are given above. 
Using \refeqn{eq.Napoleon}  and \refeqn{eq.NapoleonReflection}, the  closed-form expression of the double outer Napolean transformation  $\NT{-}{2}$  can be obtained as:
\begin{lemma} \label{lem.DoubleNapoleon}
An arbitrary triple, $\vectbf{x} = \tr{\threevecT{\vect{x}_1}{\vect{x}_2}{\vect{x}_3}} \in \R^{3d}$ 
gives rise to the double outer Napoleon triangle, $\NT{-}{2} : \R^{3d} \rightarrow \R^{3d }$, according to the formula
\begin{equation} \label{eq.doubleNapoleon}
\NT{-}{2}\prl{\vectbf{x}} = \frac{2}{3}\vectbf{x} + \frac{1}{3}\TT{+}{}\prl{\vectbf{x}}.
\end{equation}
\end{lemma}
\begin{proof}
By Napoleon's theorem, $\NT{-}{}\prl{\vectbf{x}}$ is an equilateral triangle. 
Using \refeqn{eq.Napoleon} and \reflem{lem.NapoleonEquilateral}, one can obtain the result as follows:

\noindent
\begin{align}
\!\!\NT{-}{2}\prl{\vectbf{x}} &\sqz{=} \NT{-}{}\prl{\NT{-}{}\prl{\vectbf{x}}\!} \sqz{=} 2 \sqz{\cdot} \mat{1}_3 \sqz{\otimes} \ctrd{x} \sqz{-} \NT{-}{}\prl{\vectbf{x}} \sqz{=} 2 \sqz{\cdot} \mat{1}_3 \sqz{\otimes} \ctrd{x} \sqz{-} \frac{1}{3} \prl{\mat{K}\vectbf{x} \sqz{+} \TT{-}{}\prl{\vectbf{x}}\!}\!,\! \!\!\\
& \hspace{-5mm}\sqz{=} \frac{2}{3}\prl{\mat{K}\vectbf{x} \sqz{+} \vectbf{x}}  \sqz{-} \frac{1}{3} \prl{\mat{K}\vectbf{x} \sqz{+} \TT{-}{}\prl{\vectbf{x}}\!} \sqz{=} \frac{2}{3}\vectbf{x} \sqz{+}\frac{1}{3} \prl{\mat{K}\vectbf{x} \sqz{-} \TT{-}{}\prl{\vectbf{x}}\!}
\sqz{=} \frac{2}{3}\vectbf{x} \sqz{+} \frac{1}{3}\TT{+}{}\prl{\vectbf{x}}\!,\!\!\!
\end{align}
where $\mat{K}$ is defined as in \refeqn{eq.KLmat}.
\end{proof}

\noindent Notice that  $\NT{-}{2}\prl{\vectbf{x}}$ is a convex combination of  $\vectbf{x}$ and  $\TT{+}{}\prl{\vectbf{x}}$, see \reffig{fig.NapoleonTorricelli}.

\bigskip

\begin{figure}
\centering
\vspace{5mm}
\includegraphics[width=0.8 \textwidth]{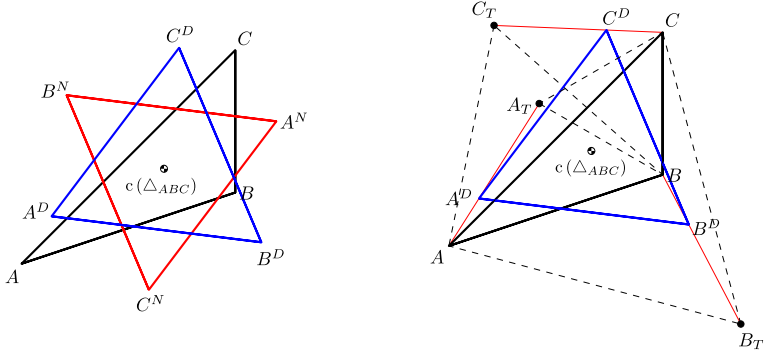} 
\caption{(left) Outer, $\triangle_{A^NB^NC^N}$,  and double outer, $\triangle_{A^DB^DC^D}$, Napoleon transformations of a triangle $\triangle_{ABC}$.
(right) The double outer Napoleon triangle $\triangle_{A^DB^DC^D}$ is a convex combination of the original triangle $\triangle_{ABC}$ and its inner Torricelli configuration $\triangle_{A_TB_TC_T}$. }
\label{fig.NapoleonTorricelli}
\end{figure}

\section{Optimality of Napoleon Transformations.}
\label{app.NapoleonOptimality}

To best of our knowledge, the Napoleon transformation $\NT{\pm}{}$ is mostly recognized as being a function into the space of equilateral triangles.
In addition to this inherited property, $\NT{\pm}{2}$ has an optimality property that is not immediately obvious.
Although the double inner Napoleon transformation $\NT{+}{2}$ is not  really  that interesting to work with, it gives a hint about the optimality of $\NT{-}{2}$: for any given triangle $\NT{+}{2}$ yields a trivial triangle all of whose vertices are located at the centroid of the given triangle which, by definition, minimizes the sum of squared distances to the vertices of the original triangle.      
Surprisingly, one has a similar optimality property for $\NT{-}{2}$:

\begin{theorem} \label{thm.NapoleonOptimality}
The double outer Napoleon transformation $ \NT{-}{2}\prl{\vectbf{x}}$ \refeqn{eq.doubleNapoleon} yields the equilateral triangle  most closely aligned with
$\triangle_{\vectbf{x}}$ in the sense that it minimizes the total sum of squared distances between corresponding vertices.
That is to say, for any $\vectbf{x}= \tr{\brl{\vect{x}_1, \vect{x}_2, \vect{x}_3}} \in \R^{3d}$, $\NT{-}{2}\prl{\vectbf{x}}$ is an optimal solution  of the following problem 
\begin{equation}\label{eq.NapoleonOptimization}
\begin{split}
\text{minimize } &\hspace{2mm}\sum_{i = 1}^{3}\norm{\vect{x}_i - \vect{y}_i}^2  \\
\text{subject to} & \hspace{2mm}\norm{\vect{y}_1 - \vect{y}_2}^2 = \norm{\vect{y}_1 - \vect{y}_3}^2 =  \norm{\vect{y}_2 - \vect{y}_3}^2  
\end{split}
\end{equation}
where $\vectbf{y} = \tr{\brl{\vect{y}_1, \vect{y}_2, \vect{y}_3}} \in \R^{3d}$.
Further, if $\vectbf{x}$ is non-collinear, then  \refeqn{eq.NapoleonOptimization} has a unique solution.
\end{theorem}
\begin{proof}
Using the method of Lagrange multipliers \cite{bertsekas_1999},  we first show that an optimal solution of \refeqn{eq.NapoleonOptimization} lies in the plane containing the triangle $\triangle_{\vectbf{x}}$. 
Then, to show the result, we solve  \refeqn{eq.NapoleonOptimization} using a proper parametrization of equilateral triangles in $\R^2$.

The Lagrangian formulation of \refeqn{eq.NapoleonOptimization} minimizes

\noindent
\begin{align}
L\prl{\vectbf{y}, \lambda_1, \lambda_2} & = \sum_{i = 1}^{3}\norm{\vect{x}_i- \vect{y}_i}_2^2 + \lambda_1\prl{\norm{\vect{y}_1 -\vect{y}_2 }_2^2 - \norm{\vect{y}_1 -\vect{y}_3 }_2^2} \nonumber \\
&\hspace{35mm} + \lambda_2\prl{\norm{\vect{y}_1 -\vect{y}_2 }_2^2 - \norm{\vect{y}_2 -\vect{y}_3 }_2^2},
\end{align}
where $\lambda_1, \lambda_2 \in \R$ are Lagrange multipliers.    
Necessary conditions for the locally optimal solutions of \refeqn{eq.NapoleonOptimization} is \footnote{Here, $\nabla_{\vectbf{y}}$ denotes the gradient taken with respect to the coordinate $\vectbf{y}$.}
\begin{equation}
\nabla_{\vectbf{y}} L\prl{\vectbf{y}, \lambda_1, \lambda_2} = 2\brl{
\begin{array}{l}
\prl{\vect{y}_1 - \vect{x}_1} + \lambda_1\prl{\vect{y}_3 - \vect{y}_2} + \lambda_2 \prl{\vect{y}_1 - \vect{y}_2} \\
\prl{\vect{y}_2 - \vect{x}_2} + \lambda_1\prl{\vect{y}_2 - \vect{y}_1} + \lambda_2 \prl{\vect{y}_3 - \vect{y}_1}\\
\prl{\vect{y}_3 - \vect{x}_3} - \lambda_1\prl{\vect{y}_3 - \vect{y}_1} - \lambda_2 \prl{\vect{y}_3 - \vect{y}_2}
\end{array}
} = 0,
\end{equation} 
from which one can conclude that an optimal solution of \refeqn{eq.NapoleonOptimization} lies in the plane containing $\triangle_{\vectbf{x}}$.
Accordingly, without loss of generality, suppose that $\triangle_{\vectbf{x}}$ is a positively oriented triangle in $\R^2$, i.e.   its vertices  are in counter-clockwise order in $\R^2$. 

In general, an equilateral triangle $\triangle_{\vectbf{y}}$ in $\R^2$ with vertices $\vectbf{y} = \tr{\brl{\vect{y}_1, \vect{y}_2, \vect{y}_3}} \in \R^{6}$  can be uniquely parametrized  using  two of its vertices, say $\vect{y}_1$ and $\vect{y}_2$, and a binary variable $k \in \crl{-1,+1}$ specifying the orientation of $\triangle_{\vectbf{y}}$; for instance, $k = +1$ if $\triangle_{\vectbf{y}}$ is positively oriented, and so on.  
Consequently, the remaining vertex, $\vect{y}_3$, can be located as 
\begin{equation}
\vect{y}_3 = \frac{1}{2}\prl{\vect{y}_1+\vect{y}_2} + k \frac{\sqrt{3}}{2} \RotMat_{\pi/2} (\vect{y}_2- \vect{y}_1), \label{eq.TriangleParametrization}
\end{equation}
where  $\RotMat_{\pi/2} =\twomatrix{0}{-1}{1}{0}$ is the rotation matrix defining a rotation by $\pi/2$.
%

Hence, one can rewrite the optimization problem \refeqn{eq.NapoleonOptimization}  in term of new parameters as an unconstrained optimization problem as: for $\vect{y}_1, \vect{y}_2 \in \R^2$ and $k \in \crl{-1,1}$,
\begin{equation}
\text{minimize } \hspace{2mm} \norm{\vect{x}_1 - \vect{y}_1}_2^2 + \norm{\vect{x}_2 - \vect{y}_2}_2^2 + \norm{\vect{x}_3 - \mat{M} \vect{y_1} - \tr{\mat{M}}\vect{y}_2}_2^2
 \label{eq.NapoleonOptimization2}
\end{equation}
where   $\mat{M} \ldf \frac{1}{2}\mat{I} - k \frac{\sqrt{3}}{2} \RotMat_{\pi/2}$,  and $\mat{I}$ is the $2\times 2$ identity matrix. 
Note that $\mat{M} + \tr{\mat{M}} = \mat{I}$, $\tr{\mat{M}}\mat{M} = \mat{M}\tr{\mat{M}} = \mat{I}$ and $\mat{M}^2 = -\tr{\mat{M}}$.
 
For a fixed $k \in \crl{-1,1}$,  \refeqn{eq.NapoleonOptimization2} is a convex optimization problem of $\vect{y}_1$ and $\vect{y}_2$ because every norm on $\R^n$ is convex and compositions of convex functions with affine transformations preserve convexity \cite{boyd_2004}.
Hence, a global optimal solution of \refeqn{eq.NapoleonOptimization2} occurs where the gradient of the objective function is zero at,  
\begin{equation}
\twomatrix
{\prl{\mat{I} + \tr{\mat{M}}\mat{M}}}
{\tr{\prl{\mat{M}^2}}}
{\mat{M}^2}
{\prl{\mat{I} + \mat{M}\tr{\mat{M}}}}
\twovec{\vect{y}_1}{\vect{y_2}}  = 
\twovec{\vect{x}_1 + \tr{\mat{M}}\vect{x}_3}{\vect{x}_2 + \mat{M}\vect{x}_3},
\end{equation}
which simplifies to
\begin{equation}
\twomatrix
{2\mat{I}}
{-\mat{M}}
{-\tr{\mat{M}}}  
{2\mat{I}}
\twovec{\vect{y}_1}{\vect{y}_2}  = 
\twovec{\vect{x}_1 + \tr{\mat{M}}\vect{x}_3}{\vect{x}_2 + \mat{M}\vect{x}_3}. \label{eq.NapoleonOptimality}
\end{equation} 
Note that the objective function, $f\prl{\vectbf{y}}$, is strongly convex since its Hessian, $\nabla^2f\prl{\vectbf{y}}$,  satisfies
\begin{equation}
\nabla^2f\prl{\vectbf{y}}= \twomatrix
{2\mat{I}}
{-\mat{M}}
{-\tr{\mat{M}}}  
{2\mat{I}} \succeq \mat{I},
\end{equation} 
which means that for a fixed $k \in \crl{-1, +1}$ the optimal solution of \refeqn{eq.NapoleonOptimization2} is unique.  

Now observe that
\begin{equation}
\frac{1}{3}
\twomatrix
{2\mat{I}}
{\mat{M}}
{\tr{\mat{M}}} 
{2\mat{I}}
\twomatrix
{2\mat{I}}
{-\mat{M}}
{-\tr{\mat{M}}} 
{2\mat{I}}
 = 
\twomatrix{\mat{I}}{\mat{0}}{\mat{0}}{\mat{I}},  
\end{equation}
hence the solution of linear equation  \refeqn{eq.NapoleonOptimality} is

\noindent
\begin{align}
\!\twovec{\vect{y}_1}{\vect{y}_2} &\sqz{=} \frac{1}{3}\!\twomatrix
{2\mat{I}}
{\mat{M}}
{\tr{\mat{M}}} 
{2\mat{I}}
\!\twovec{\vect{x}_1 \sqz{+} \tr{\mat{M}}\vect{x}_3}{\vect{x}_2 \sqz{+} \mat{M}\vect{x}_3} \sqz{=} \frac{1}{3}\!\twovec{2\vect{x}_1 \sqz{+} 2\tr{\mat{M}}\vect{x}_3 \sqz{+} \mat{M}\vect{x}_2 \sqz{+} \mat{M}^2 \vect{x}_3}
{\tr{\mat{M}}\vect{x}_1 \sqz{+} \tr{\prl{\mat{M}^2}}\vect{x}_3 \sqz{+} 2\vect{x}_2 \sqz{+} 2 \mat{M} \vect{x}_3}\!,\! \\
&\sqz{=} \frac{1}{3} \!\twovec{2\vect{x}_1 \sqz{+} \tr{\mat{M}}\vect{x}_3 \sqz{+}\mat{M}\vect{x}_2}{2\vect{x}_2 \sqz{+} \tr{\mat{M}}\vect{x}_1 \sqz{+} \mat{M}\vect{x}_3 }
\sqz{=} \twovec{\frac{1}{2}\prl{\vect{x}_1 \sqz{+} \frac{\vect{x}_1 + \vect{x}_2 + \vect{x}_3}{3}} \sqz{+} k\frac{1}{2\sqrt{3}} \RotMat_{\pi/2}\prl{\vect{x}_3 \sqz{-} \vect{x}_2}}
{\frac{1}{2}\prl{\vect{x}_2 \sqz{+} \frac{\vect{x}_1 + \vect{x}_2 + \vect{x}_3}{3}} \sqz{+} k\frac{1}{2\sqrt{3}} \RotMat_{\pi/2}\prl{\vect{x}_1 \sqz{-} \vect{x}_3}}\!.\!
\end{align}
Here, substituting $\vect{y}_1$ and $\vect{y}_2$ back into \refeqn{eq.TriangleParametrization} yields
\begin{equation}
\vect{y}_3 = \frac{1}{2}\prl{\vect{x}_3 + \frac{\vect{x}_1 + \vect{x}_2 + \vect{x}_3}{3}} + k\frac{1}{2\sqrt{3}} \RotMat_{\pi/2}\prl{\vect{x}_2 - \vect{x}_1} .
\end{equation}
Thus, overall, we have
\begin{equation}
\vectbf{y} \sqz{=} \frac{2}{3}\vectbf{x} \sqz{+} \frac{1}{3}\prl{\frac{1}{2}\mat{K}\vectbf{x} \sqz{+} k\frac{\sqrt{3}}{2}\prl{\mat{I}_3 \otimes \mat{R}_{\pi/2}}\mat{L} \vectbf{x}} \sqz{=} \left \{
\begin{array}{@{}l@{}l@{}}
\frac{2}{3}\vectbf{x} + \frac{1}{3}\TT{+}{}\prl{\vectbf{x}} & \text{, if } k = +1, \\[1mm]
\frac{2}{3}\vectbf{x} + \frac{1}{3}\TT{-}{}\prl{\vectbf{x}} & \text{, if } k = -1.
\end{array}
\right. \!\!
\end{equation}
where $\mat{K}$ and $\mat{L}$ are defined as in \refeqn{eq.KLmat}.
Recall that $\triangle_{\vectbf{x}}$ is assumed to be positively oriented, i.e.  $\RotMat_{\vectbf{x}} = \RotMat_{\pi/2}$, and so  it is convenient to have the results in terms of Torricelli transformations $\TT{\pm}{}$ \refeqn{eq.Torricelli}.
As a result, the difference of $\vectbf{y}$ and $\vectbf{x}$ is simply given by
\begin{equation}
\vectbf{y} - \vectbf{x} = \left \{
\begin{array}{ll}
\frac{1}{3}\prl{\TT{+}{}\prl{\vectbf{x}} - \vectbf{x}}  \text{, if } k = +1, \\[1mm]
\frac{1}{3}\prl{\TT{-}{}\prl{\vectbf{x}} - \vectbf{x}}  \text{, if } k = -1.
\end{array}
\right .
\end{equation}

Finally, one can easily verify that the optimum value of $k$ is equal to $+1$ since the distance of  $\vectbf{x}$ to  its inner Torricelli configuration $\TT{+}{}\prl{\vectbf{x}}$ is always less than or equal to its distance to the outer Torricelli configuration $\TT{-}{}\prl{\vectbf{x}}$. 
Here, the equality only holds if $\vectbf{x}$ is collinear.
Thus, an optimal solution of \refeqn{eq.NapoleonOptimization} coincides with the double outer Napoleon transformation, $\NT{-}{2}\prl{\vectbf{x}}$ \refeqn{eq.doubleNapoleon}, and it is the  unique solution of \refeqn{eq.NapoleonOptimization} if $\vectbf{x}$ is non-collinear. 
\end{proof}

As a final remark we would like to note that our particular interest in the optimality of Napoleon triangles comes from our research on coordinated robot navigation, where a group of robots require to interchange their (structural) adjacencies through a minimum cost configuration determined by the double outer Napoleon transformation~\cite{arslan_guralnik_kod_WAFR2014}.

\begin{acknowledgment}{Acknowledgment.}
We would like to thank Dan P. Guralnik for the numerous discussions and kind feedback.
This work was funded in part by the Air Force Office of Science
Research under the MURI FA9550-10-1-0567.
\end{acknowledgment}

\bibliographystyle{ieeetr}
\bibliography{napoleon}

\vfill\eject

\end{document}

%% file: package.tex
\usepackage{comment} 
\usepackage{lineno} 
\usepackage[%
hyperfigures=true,%
backref=page,%
pagebackref=true,%
breaklinks=true,%
colorlinks=true,%
citecolor=green,%
linkcolor=blue%
]{hyperref} 

\usepackage{pifont}
\usepackage{amsmath}
\usepackage{amsfonts}
\usepackage{amssymb}
\usepackage{amsthm}
\usepackage{graphicx}
\usepackage{color}
\usepackage{url}

\usepackage{epsfig} 
\usepackage{psfrag}

%% file: header.tex
	\theoremstyle{plain}

	\newtheorem{theorem}{Theorem}
	\newtheorem{lemma}{Lemma}

	\theoremstyle{definition}

	\newtheorem{remark}{Remark}
 
    \newcommand{\prl}[1] 		{\!\left(#1\right)}
    \newcommand{\brl}[1] 		{\left[ #1 \right] }
    \newcommand{\crl}[1] 		{\left\{ #1 \right\} }
    
    \newcommand{\R}             {\mathbb{R}} 
 	\newcommand{\Sp}        		{\mathbb{S}} 

	\newcommand{\norm}[1] 		{\left \| #1 \right \|}

    \newcommand{\vectprod}[2]    { \tr{#1} #2}
    \newcommand{\tr}[1]         {{#1}^\mathrm{T}}
     
	\newcommand{\mat}[1] 		{\mathrm{\mathbf{#1}}}

	\newcommand{\vect}[1] 		{\mathrm{#1}}
	\newcommand{\vectbf}[1]		{\mathrm{\mathbf{#1}}}

    \newcommand{\twovec} [2] { \left[ \begin{array}{@{}c@{}} 
          #1 \\ #2 \end{array} \right] }
    
    \newcommand{\twovect} [2] { [\;#1, \;#2\;] }

    \newcommand{\threevecT} [3] { \left[ \begin{array}{@{}c@{}@{}c@{}@{}c@{}} 
          #1, & #2, & #3 \end{array} \right] }

    \newcommand{\twomatrix}[4]{ \left[\begin{array}{@{}cc@{}}
                                   #1 & #2\\
                                   #3 & #4\\
                                 \end{array}
                           \right]}

    \newcommand{\refeqn}[1]			{(\ref{#1})}
    \newcommand{\reffig}[1]			{Figure \ref{#1}}

	\newcommand{\reflem}[1]			{Lemma \ref{#1}}

    \newcommand{\ldf}				{  \: {\mathbf := }  \: }

    \newcommand{\sqz}[1]            {\! #1 \!}



%% file: notation.tex
\newcommand{\ctrd}[1]			{\mathrm{c}\prl{#1}} 
\newcommand{\NT}[2]              {\mathrm{N}_{#1}^{#2}}  
\newcommand{\TT}[2]              {\mathrm{T}_{#1}^{#2}}  

\newcommand{\TpS}[1]            {\mat{P}\prl{#1}} 
\newcommand{\RotMat}            {\mat{R}} 

%% file: body.bbl
\begin{thebibliography}{10}

\bibitem{coxeter1996geometry}
H.~S.~M. Coxeter and S.~L. Greitzer, {\em Geometry revisited}, vol.~19.
\newblock Mathematical Association of America, 1996.

\bibitem{grunbaum_AMM2012}
B.~Gr{\"u}nbaum, ``Is {N}apoleon's theorem really {N}apoleon's theorem?,'' {\em
  The American Mathematical Monthly}, vol.~119, no.~6, pp.~495--501, 2012.

\bibitem{wetzel_AMM1992}
J.~E. Wetzel, ``Converses of {N}apoleon's theorem,'' {\em The American
  Mathematical Monthly}, vol.~99, no.~4, pp.~339--351, 1992.

\bibitem{rigby_MM1991}
J.~Rigby, ``Napoleon, {E}scher, and tessellations,'' {\em Mathematics
  Magazine}, vol.~64, no.~4, pp.~242--246, 1991.

\bibitem{kupitz_martini_spirova_JOTA2013}
Y.~S. Kupitz, H.~Martini, and M.~Spirova, ``The {F}ermat--{T}orricelli problem,
  part i: A discrete gradient-method approach,'' {\em Journal of Optimization
  Theory and Applications}, pp.~1--23, 2013.

\bibitem{gueron_tessler_AMM2002}
S.~Gueron and R.~Tessler, ``The {F}ermat--{S}teiner problem,'' {\em The
  American Mathematical Monthly}, vol.~109, no.~5, pp.~443--451, 2002.

\bibitem{martini_weissbach_GD1999}
H.~Martini and B.~Weissbach, ``Napoleon's theorem with weights in n-space,''
  {\em Geometriae Dedicata}, vol.~74, no.~2, pp.~213--223, 1999.

\bibitem{hajja_martini_spirova_BAG2008}
M.~Hajja, H.~Martini, and M.~Spirova, ``New extensions of {N}apoleon's theorem
  to higher dimensions,'' {\em Beitr. Algebra Geom}, vol.~49, no.~1,
  pp.~253--264, 2008.

\bibitem{Graham_1981}
A.~Graham, {\em Kronecker products and matrix calculus: with applications},
  vol.~108.
\newblock Horwood Chichester, 1981.

\bibitem{bertsekas_1999}
D.~P. Bertsekas, {\em Nonlinear Programming}.
\newblock Athena Scientific, 1999.

\bibitem{boyd_2004}
S.~P. Boyd and L.~Vandenberghe, {\em Convex Optimization}.
\newblock Cambridge University Press, 2004.

\bibitem{arslan_guralnik_kod_WAFR2014}
O.~Arslan, D.~Guralnik, and D.~E. Koditschek, ``Navigation of distinct
  euclidean particles via hierarchical clustering,'' in {\em Algorithmic
  Foundations of Robotics XI}, vol.~107 of {\em Springer Tracts in Advanced
  Robotics}, pp.~19--36, 2015.

\end{thebibliography}
